\documentclass[12pt]{amsart}
\usepackage{hyperref}
\usepackage{amsmath}
\usepackage{amsthm}
\usepackage{amsfonts}
\usepackage{amssymb}
\usepackage{graphicx}
\DeclareGraphicsExtensions{.eps}

\newtheorem{theorem}{Theorem}[section]
\newtheorem{utv*}{Proposition}
\newtheorem{hyp*}{Conjecture}
\newtheorem{lemma}[theorem]{Lemma}

\newtheorem{defin}{Definition}
\newtheorem{zamech}{Remark}
\newtheorem*{th*}{Theorem}

\def\sli{\sum\limits}
\def\ili{\int\limits}

\def\a{\alpha}
\def\R{\mathbb{R}}

\def\ep{\varepsilon}

\input epsf.sty
\begin{document}

\title{Random ``dyadic'' lattice in geometrically doubling metric space and $A_2$ conjecture}
\author{Alexander Reznikov, Alexander Volberg}
\address{Department of Mathematics, Michigan State University, East Lansing, MI 48824, USA}
\email{reznikov@ymail.com, sashavolberg@yahoo.com}
\date{}
\maketitle
\begin{abstract}
Recently three proofs of the $A_2$-conjecture were obtained. All of them are ``glued'' to euclidian space and a special choice of one random dyadic lattice. We build a random ``dyadic'' lattice in any doubling metric space which have properties that are enough to prove the $A_2$-conjecture in these spaces.
\end{abstract}
\section{Introduction}
Out goal is to build an analog of the probability space of dyadic lattices, which was first constructed in ~\cite[p. 207]{NTV}. The main property of this space was that probability of a fixed cube to be ``good'' in a very precise sense is bounded away from zero.

Later in proofs, related to weighted Calderon-Zygmund theory, it was enough to consider a probability space whose elementary event was a pair of dyadic lattices. This ideology was used in ~\cite{NTV2} and ~\cite{NTV}. In the case of a geometrically doubling metric space (GDMS) Michael Christ, ~\cite{Chr} introduced dyadic ``lattice'', and later Hyt\"onen and Martikainen, ~\cite{HM}, randomized Christs construction. They also have a probability space, which consists of pairs of dyadic lattices on the GDMS. The definition of a ``good'' cube in one lattice was given using the second lattice.

However, recently the solution of so called $A_2$ conjecture in Euclidian setting essentially required the construction from ~\cite[p. 207]{NTV}, where the elementary event is {\bf one} dyadic lattice, and not a pair of them. The definition of a ``good'' cube in this lattice was given in terms of other cubes of the {\bf same} lattice. All the proofs of $A_2$ conjecture so far were based on decomposition of a Calderon-Zygmund operator to dyadic shifts and thus required this {\bf one} lattice randomization.

The first proof of $A_2$ conjecture was given by Hyt\"{o}nen, in which he essentially used reduction to the weighted $T1$ theorem from ~\cite{PTV}. The new tool --- in comparison to ~\cite{NTV} --- is that one can, instead of estimating the contribution of ``bad'' cubes (as in ~\cite{NTV}), just ignore this contribution completely. For that one needed (we repeat) the construction of probability space, consisting of {\bf one} dyadic lattice as an elementary event. Also one needs that the probability to be ``good'' is the same for all cubes.

Immediately after this proof a simplified proof was given in ~\cite{HPTV}. It again uses the same properties of probability space of dyadic lattices, but the decomposition of the operator to dyadic shifts was simplified, and instead of the weighted $T1$ theorem for an arbitrary Calderon-Zygmund operator (~\cite{PTV}) one used $T1$ theorem for shifts, \cite{NTV3}.

Another Bellman function proof of the same $A_2$ conjecture was recently obtained in ~\cite{NV}.

In the present work we build a probability space of dyadic ``lattices'' in a GDMS, where an elementary event is one lattice, and the probability to be ``good'' is the same for all cubes in the lattice. As an application, with a combination of this idea and tools from any of ~\cite{Hyt}, ~\cite{HPTV}, or ~\cite{NV} (the last one will give the shortest proof), we obtain a proof of the $A_2$ conjecture in an arbitrary geometrically doubling metric space.

\section{Acknowledgements}
We are very grateful to Michael Shapiro and Dapeng Zhan for their help in proving the main lemma. We want to express our gratitude to Jeffrey Schenker for valuable discussions.
\section{First step}
Consider a doubling metric space $X$ with metric $d$ and doubling constant $A$. Instead of $d(x,y)$ we write $|xy|$.

As authors of ~\cite{HM}, we essentially use the idea of Michael Christ, ~\cite{Chr}, but randomize his construction in a different way. Therefore, we want to guard the reader that even though on the surface the proof below is very close to the proof from ~\cite{HM}, however, our construction is essentially different, and so the proof of the assertion in our main lemma, which was not hard in ~\cite{HM}, becomes much more subtle here.

For a number $k>0$ we say that a set $G$ is a $k$-grid if $G$ is maximal set, such that for any $x,y\in G$ we have $d(x,y)\geqslant k$.

Take a small positive number $\delta$ and a large natural number $N$, and for every $M\geqslant N$ fix $G_M=\{z_M^\a\}$ --- a $\delta^M$-grid. Now take $G_N$ and randomly choose a $G_{N-1}=\delta^{N-1}$-grid in $G_N$. Then take $G_{N-1}$ and randomly choose a $G_{N-2}=\delta^{N-2}$-grid there.
\begin{lemma}
$$
\bigcup\limits_{y\in G_{N-k}} B(y, 3\delta^{N-k}) = X.
$$
(for $N+k$ this is obvious)
\end{lemma}
\begin{proof}
Take $x\in X$ 

Then, since $G_{N}$ is maximal, there exists a point $y_0 \in G_{N}$, such that $|xy_0|\leqslant \delta^{N}$. Since $G_{N-1}$ is maximal in $G_N$, there is a point $y_1\in G_{N-1}$, such that $|y_0y_1|\leqslant \delta^{N-1}$. Similarly we get $y_2, \ldots, y_k$ and then
$$
|xy_k| \leqslant |xy_0|+\ldots + |xy_k|\leqslant \delta^{N}+\ldots + \delta^{N-k} = \delta^{N-k}(1+\delta+\ldots + \delta^{k})\leqslant \frac{\delta^{N-k}}{1-\delta} \leqslant 2\delta^{N-k}.
$$
\end{proof}

Once we have all our sets $G_N$, we introduce a relationship $\prec$ between points. We follow ~\cite{HM} and ~\cite{Chr}.

Take a point $y_{k+1}\in G_{k+1}$. There exists at most one $y_k\in G_k$, such that $|y_{k+1}y_k|\leqslant \frac{\delta^{k}}{4}$. This is true since if there are two such points $y_k^1, \; y_k^2$, then
$$
|y_k^1 y_k^2|\leqslant \frac{\delta^{k}}{2},
$$
which is a contradiction, since $G_k$ was a $\delta^{k}$-grid in $G_{k+1}$.

Also there exists at least one $z_k\in G_k$ such that $|y_{k+1} z_k|\leqslant 3 \delta^{k}$. This is true by the lemma.

Now, if there exists an $y_k$ as above, we set $y_{k+1}\prec y_k$. If no, then we pick one of $z_k$ as above and set $y_{k+1}\prec z_k$. For all other $x\in G_{k}$ we set $y_{k+1}\not\prec x$. Then extend by transitivity.

We also assume that $y_k\prec y_k$.

We do this procedure randomly and independently, and treat same families of $G_k$'s with different $\prec$-law as different families.

Take now a point $y_k\in G_k$ and define
$$
Q_{y_k} = \bigcup\limits_{z\prec y_k, z\in G_{\ell}} B(z, \frac{\delta^{\ell}}{100}).
$$

\begin{lemma}
For every $k$ we have
$$
X=\bigcup\limits_{y_k\in G_k} \textup{clos}(Q_{y_k})
$$
\end{lemma}
\begin{proof}
Take any $x\in X$. By the previous lemma, for every $m\geqslant k$ there exists a point $x_m\in G_m$, such that $|xx_m|\leqslant 3\delta^{m}$. In particular, $x_m \to x$. Fix for a moment $x_m$. Than there are points $y_{m-1}\in G_{m-1}, \ldots, y_k\in G_{k}$, such that $x_m \prec y_{m-1}\prec \ldots \prec y_k$. In particular, $x_m \in Q_{y_k}$, where $y_k$ depends on $x_m$. Than
$$
|y_k x| \leqslant |y_k x_m| + |x_m x|\leqslant |y_k x_m| + 3\delta^{m} \leqslant |y_k x_m| + 3\delta^{k}.
$$
Moreover, by the chain of $\prec$'s, we know that $|y_k x_m|\leqslant 10\delta^{k}$. Therefore,
$$
|y_k x|\leqslant 15\delta^{k}.
$$
We claim that the set $\{y_k\}=\{y_k(x_m)\}_{m\geqslant k}$ is finite. This is true since all $y_k$'s are separated from each other and by the doubling of our space (we are ``stuffing'' the ball $B(x, 15\delta^{k})$ with balls $B(y_k, \delta^{k})$).

So, take an infinite subsequence $x_{m}$ that corresponds to one point $y_k\in G_k$. Then we get $x_m \in Q_{y_k}$, $x_m\to x$, so $x\in \textup{clos} Q_{y_k}$, and we are done.
\end{proof}

\section{Second step: technical lemmata}
Define
$$
\tilde{Q}_{y_k} = X\setminus \bigcup_{z_k\not= y_k, z_k\in G_k}\textup{clos}Q_{z_k}.
$$
In particular,
$$
Q_{y_k}\subset \tilde{Q}_{y_k}\subset \textup{clos}(Q_{y_k}).
$$
\begin{lemma}[Lemma 4.5 in ~\cite{HM}]
Let $m$ be a natural number, $\ep>0$, and $\delta^m \geqslant 100\ep$. Suppose $x\in \textup{clos}Q_{y_k}$ and $dist(x, X\setminus \tilde{Q}_{y_k})<\ep \delta^{k}$. Then for any chain
$$
z_{k+m}\prec z_{k+m-1}\prec\ldots \prec z_{k+1}\prec z_k,
$$
such that $x\in \textup{clos}Q_{z_{k+m}}$, there holds
$$
|z_i z_j|\geqslant \frac{\delta^{j}}{100}, \; \; \; k\leqslant j < i \leqslant k+m.
$$

\end{lemma}
\begin{proof}
Suppose $|z_i z_j| < \frac{\delta^{j}}{100}$. We first consider a case when $z_k=y_k$.
Since $z_j\prec z_k=y_k$, we have $B(z_j, \frac{\delta^j}{200})\subset Q_{y_k}\subset \tilde{Q}_{y_k}$. Therefore,
$$
\frac{\delta^j}{200}\leqslant dist(z_j, X\setminus \tilde{Q}_{y_k})\leqslant dist(x, X\setminus \tilde{Q}_{y_k}) + dist(x, z_i) + dist(z_i, z_j) < \ep \delta^{k} + 5\delta^i + \frac{\delta^{j}}{100}
$$
If $\delta$ is less than, say, $1\over 1000$, then we get a contradiction.

The only not obvious estimate is that $dist(x, z_i)<5\delta_i$. It is true since $x\in \textup{clos}Q_{z_{k+m}}$.

We have proved the lemma with assumption that $z_k=y_k$. Let us get rid of this assumption. We know that
$$
x\in \textup{clos}Q_{z_{k+m}}\subset \textup{clos}Q_{z_{k}}.
$$
Also we have $x\in \textup{clos}Q_{y_k}$, so, since
$$
\tilde{Q}_{z_k}=X\setminus \bigcup_{u_k\not = z_k} \textup{clos}Q_{u_k} \subset X\setminus \textup{clos}Q_{y_k},
$$
we get $x\in X\setminus \tilde{Q}_{z_k}$. In particular, $dist(x, X\setminus \tilde{Q}_{z_k})=0<\ep \delta^k$, and we are in the situation of the first part. This finishes our proof.
\end{proof}

\begin{lemma}[MAIN LEMMA]
Fix $x_k\in G_k$. Then
\begin{equation}
\label{ryadom}
\mathbb{P}(\exists x_{k-1}\in G_{k-1}\colon |x_k x_{k-1}|<\frac{\delta^{k-1}}{1000})\geqslant a
\end{equation}
for some $a\in (0,1)$.
\end{lemma}
\begin{proof}
To illustrate the proof we consider a particular case, but with a stronger result.

Consider $X$ to be a tree with a distance between to neighbor vertices equal to one. The geometrically doubling condition can be expressed in the following way: each vertex has no more that $C$ sons, where $C$ is a fixed number. For the sake of simplicity we consider $C=3$.

We now randomly choose a $2$-grid $G$. In other words, we color our tree in two colors, red and green, such that
\begin{itemize}
\item If a vertex is red then all its neighbors are green;
\item The set of red vertices is maximal, i.e., if a vertex is green there is a red vertex on distance $1$.
\end{itemize}

We proof a stronger result: fix a vertex $z$ then
$$
\mathbb{P}(\exists x\in G \colon x=z) > a.
$$

Without loss of generality, we consider that $z$ is a root and it has $3$ sons.
Consider $z$ and assume that it is red. We introduce a procedure of recoloring our tree such that $z$ becomes green. If $z$ is red then all its sons, $z_1, z_2, z_3$, are green. If one of them has only green sons then everything is easy: we color $z$ in green color and this son in red.

Suppose all $z_{1,2,3}$ have one red son. Then we proceed by induction (the base, when height of our tree is $2$, is trivial). Take $z_1$ and its sons $z_{11}, z_{12}, z_{13}$. If $z_{11}$ is red then we recolor all the sub-tree of $z_{11}$, such that $z_{11}$ becomes green. We do the same thing with $z_{12}$ and $z_{13}$. After that we can color $z_{1}$ in red and $z$ in green.

Notice that this procedure has the following properties: for two different initial coloring in gives different coloring. So we proved that
$$
\mathbb{P}(z \;\mbox{is green}\: ) \geqslant \frac{1}{2}.
$$
Now,
$$
\mathbb{P}(z \; \mbox{is red}\:) = \mathbb{P}(z_{1,2,3} \; \mbox{are green}\:),
$$
so
\begin{multline}
\frac{\#\{\mbox{colorings with red} \; z\}}{\# \{\mbox{all colorings}\}} = \frac{\#\{\mbox{colorings with green} \; z_1\}}{\# \{\mbox{all colorings of  $z_1$-subtree}\}}\frac{\#\{\mbox{colorings with green} \; z_2\}}{\# \{\mbox{all colorings of $z_2$-subtree}\}}\cdot \\ \frac{\#\{\mbox{colorings with green} \; z_3\}}{\# \{\mbox{all colorings of $z_3$-subtree}\}} \cdot \\ \frac{\#\{\mbox{all colorings of $z_1$-subtree}\}\cdot \#\{\mbox{all colorings of $z_2$-subtree}\} \cdot \#\{\mbox{all colorings of $z_3$-subtree}\}}{\#\{\mbox{all colorings}\}}
\end{multline}
First three fractions are bigger than $\frac{1}{2}$ by the recoloring argument. The last one is not equal to one (getting a coloring the whole tree, we may not get a proper coloring of, say $z_1$-subtree), but is also bigger than $\frac{1}{2}$. Here is the reason:
$$
\#\{\mbox{all colorings}\} \leqslant 2\#\{\mbox{colorings with green $z$}\},
$$
and colorings with green $z$ always give a proper coloring of $z_{1,2,3}$-subtrees.

So our probability is bigger than $\frac{1}{16}$, which finishes the proof.
\end{proof}

\begin{zamech}
All numbers from this proof, such as distances $1$ and $2$, number of sons, can be changed to arbitrary fixed numbers.
\end{zamech}

After this illustration let us give the proof of the main lemma in full generality.

\begin{proof}
So let us be in a compact metric situation. By rescaling we can think that we work with $G_1$ and choose $G_0$. We can even think that the metric space consists of finitely many points, it is $X:=G_2$. The finite set $G_1\subset X$ consists of points  having the following properties:

\noindent 1. $\forall x,y\in G_1$ we have $|xy|\ge \delta$;

\noindent 2. if $z\in X\setminus G_1$ then $\exists x\in G_1$ such that $|zx|<\delta$.

These two properties are equivalent to saying that the subset $G_1$ of $X$ consists of points such that $\forall x,y\in G_1$ we have $|xy|\ge \delta$ and we cannot add any point from $X$ to $G_1$ without violating that property. In other words: $G_1$ is a {\it maximal} set with property 1.

Here the word ``maximal'' means maximal with respect to inclusion, not maximal in the sense of the number of elements.

Now we consider the new metric space $Y=G_1$ and $G_0$ is any maximal subset such that
\begin{equation}
\label{max1}
\forall x, y\in G_0\,,\, |xy|\ge 1\,.
\end{equation}
In other words, we have
\noindent 1. $\forall x,y\in G_0$ we have $|xy|\ge 1$;

\noindent 2. if $z\in Y\setminus G_0$ then $\exists x\in G_0$ such that $|zx|<1$.

There are finitely many such maximal subsets $G_0$ of $Y$. We prescribe for each choice the same probability.
Now we want to prove the claim that is even stronger than \eqref{ryadom}. Namely, we are going to prove that given $y\in Y$

\begin{equation}
\label{ryadom1}
\mathbb{P}(\exists x_{0}\in G_{0}\colon x_0 =y)\geqslant a\,,
\end{equation}
where $a$ depends only on $\delta$ and the constants of geometric doubling of our compact metric space.

Let $Y$ be any metric space with finitely many elements. We will color the points of $Y$ into red and green colors. The coloring is called proper if

\noindent 1. every red point does not have any other red point at distance $<1$;

\noindent 2. every green point has at least one red point at distance $<1$.

Given {\it a proper coloring} of $Y$ the collection of red points is called $1$-{\it lattice}. It is a maximal (by inclusion) collection of points at distance $\ge 1$ from each other.

What we need to finish the main lemma's proof is

\begin{lemma}
\label{finiteY}
Let $Y$ be a finite metric space as above. Assume $Y$ has the following property:
\begin{equation}
\label{finite}
\text{In every ball of radius  less than}\,\, 1\,\,\text{ there are at most}\,\, d\,\,\text{ elements}\,.
\end{equation}
 Let $\mathcal{L}$ be a collection of $1$-lattices in $Y$. Elements of $\mathcal{L}$ are called $L$. Let $v\in Y$. Then
$$
\frac{\text{the number of 1-lattices L such that v belongs to L}}{\text{the total number of 1-lattices L}} \ge a>0\,,
$$
where $a$ depends only on $d$.
\end{lemma}

\begin{proof}
Given $v\in Y$ consider all subsets of $B(v,1)\setminus{v}$, this collection is called $\mathcal{S}$. Let $S\in \mathcal{S}$. We call $W_S$ the collection of all proper colorings such that $v$ is green, all elements of $S$ are red, and all elements of $B(v,1) \setminus S$ are green. We call $\tilde S$ all points in $Y$, which are not in $B(v,1)$, but at distance $<1$ from some point in $S$.

All proper colorings of $Y$ such that $v$ is red are called $B$. Let us show that
\begin{equation}
\label{2d}
\text{card}\, W_S\le \text{card}\, B\,.
\end{equation}
Notice that if \eqref{2d} were proved, we would be done with Lemma \ref{finiteY}, $a\ge 2^{-d+1}$, and, consequently, the proof of the main lemma would be finished, $a\ge 2^{-\delta^{-D}}$, where $D$ is a geometric doubling constant.

To prove \eqref{2d} let us show that we can recolor any proper coloring from $W_S$ into the one from $B$, and that this map is injective. Let $L\in W_S$. We

\noindent 1. Color $v$ into red;

\noindent 2. Color $S$ into green;

\noindent 3. Elements of $\tilde S$ were all green before. We leave them green, but we find among them all those $y$ that now in the open ball $B(y,1)$  in $Y$ all elements are green. We call them yellow (temporarily) and denote them
$Z$;

\noindent 4. We enumerate $Z$ in any way (non-uniqueness is here, but we do not care);

\noindent 5. In the order of enumeration color yellow points to red, ensuring that we skip recoloring of a point in $Z$ if it is at $<1$  distance to any previously colored yellow-to-red point from $Z$. After several steps all green and yellow elements of $\tilde S$ will have the property that at distance $<1$ there is a red point;

\noindent 6. Color the rest of yellow (if any) into green and stop.

We result in a proper coloring (it is easy to check), which is obviously $B$. Suppose $L_1, L_2$ are two different proper coloring in $W_S$. Notice that the colors of $v, S, B(v,1)\setminus S$, $\tilde S$ are the same for them. So they differ somewhere else. But our procedure does not touch ``somewhere else". So the modified colorings $L_1', L_2'$ that we obtain after the algorithm 1-6 will differ as well may be even more). So our map $W_S\rightarrow B$ (being not uniquely defined) is however injective. We proved \eqref{2d}.

\end{proof}

\end{proof}

\section{Main definition and theorem}
\begin{defin}[Bad cubes]
Take a ``cube'' $Q=Q_{x_k}$. We say that $Q$ is good if there exists a cube $Q_1=Q_{x_n}$, such that
if
$$
\delta^k \leqslant \delta^{r} \delta^n \; \; \; (k\geqslant n+r)
$$
then either
$$
dist(Q, Q_1)\geqslant \delta^{k\gamma}\delta^{n(1-\gamma)}
$$
or
$$
dist(Q, X\setminus Q_1)\geqslant \delta^{k\gamma}\delta^{n(1-\gamma)}.
$$
\end{defin}
\begin{zamech}
Notice that $\delta^k=\ell(Q)$ just by definition.
\end{zamech}

If $Q$ is not good we call it bad.
\begin{theorem}
Fix a cube $Q_{x_k}$. Then
$$
\mathbb{P}(Q_{x_k}\; \mbox{is bad}\:) \leqslant \frac{1}{2}.
$$
\end{theorem}
\begin{zamech}[Discussion]
This theorem makes sense because when we fix a cube $Q_k$, say, $k\geqslant N$, so the grid $G_k$ is not even random, we can make big cubes random. And we claim that for big quantity of choices, our big cubes will have $Q_k$ either ``in the middle'' or far away, but not close to the boundary.
\end{zamech}

\begin{defin}
For $Q=Q_k$ define
$$
\delta_{Q}(\ep)=\delta_{Q}=\{x\colon dist(x, Q)\leqslant \ep\delta^{k} \;\mbox{and}\; dist(x, X\setminus Q)\leqslant \ep\delta^{k}\}
$$
\end{defin}

\begin{lemma}
$$
\mathbb{P}(x\in \delta_{Q_k}) \leqslant \ep^{\eta}
$$
for some $\eta>0$.
\end{lemma}
\begin{proof}[Proof of the theorem]
Take the cube $Q_{x_k}$. There is a unique (random!) point $x_{k-s}$ such that $x_k\in Q_{x_{k-s}}$. Then
$$
dist(Q_{x_k}, X\setminus Q_{x_{k-s}})\geqslant dist(x_k, X\setminus Q_{x_{k-s}}) - diam(Q_{x_k}) \geqslant dist(x_k, X\setminus Q_{x_{k-s}}) - C \delta^{k}.
$$
Assume that $dist(x_k, X\setminus Q_{x_{k-s}})>2\delta^{k\gamma}\delta^{(k-s)(1-\gamma)}$ and that $s\geqslant r$ (this assumption is obvious, otherwise $Q_{x_{k-s}}$ does not affect goodness of $Q_{x_k}$).

Then, if $r$ is big enough ($\delta{r(1-\gamma)}<\frac{1}{2}$) we get
$$
dist(Q_{x_k}, X\setminus Q_{x_{k-s}}) \geqslant \delta^{k\gamma}\delta^{(k-s)(1-\gamma)},
$$
and so $Q_{x_k}$ is good.
Therefore,
$$
\mathbb{P}(Q_{x_k} \; \mbox{is bad}\:) \leqslant C \sli_{s\geqslant r} \mathbb{P}(x_k\in \delta_{Q_{k-s}}(\ep=\delta^{s\gamma}))\leqslant C \sli_{s\geqslant r} \delta^{\eta \gamma s}\leqslant 100C \delta^{\eta \gamma r}.
$$
For sufficiently large $r$ (or small $\delta$) this is less than $1\over 2$.
\end{proof}
\begin{zamech}[Discussion]
At the end of the proof we have claimed that
$$
\mathbb{P}(Q_{x_k} \; \mbox{is bad}\:) \leqslant C \sli_{s\geqslant r} \mathbb{P}(x_k\in \delta_{Q_{k-s}}(\ep=\delta^{s\gamma}))\leqslant C \sli_{s\geqslant r} \delta^{\eta \gamma s}\leqslant 100C \delta^{\eta \gamma r}.
$$
In particular, we did some estimate of the probability that $x_{k}\in \delta_{Q_{k-s}}$. Here it is crucial that the point $x_k$ is fixed and not random, so in some sense cubes $Q_{k-s}$ does not depend on $x_k$ (the conditional probability is equal to the unconditional probability, since $Q_k$ is fixed and not random).
\end{zamech}
\begin{proof}[Proof of the lemma]
Fix the largest $m$ such that $500\ep \leqslant \delta^{m}$. Choose a point $x_{k+m}$ such that $x\in Q_{x_{k+m}}$. Then by the main lemma
$$
\mathbb{P}(\exists x_{k+m-1}\in G_{k+m-1}\colon |x_{k+m} x_{k+m-1}|<\frac{\delta^{k+m-1}}{1000})\geqslant a.
$$
Therefore,
$$
\mathbb{P}(\forall x_{k+m-1}\in G_{k+m-1}\colon |x_{k+m} x_{k+m-1}|\geqslant\frac{\delta^{k+m-1}}{1000})\leqslant 1-a.
$$
Let now
$$
x_{k+m}\prec x_{k+m-1}.
$$
Then
$$
\mathbb{P}(\forall x_{k+m-2}\in G_{k+m-2}\colon |x_{k+m-1} x_{k+m-2}|\geqslant\frac{\delta^{k+m-2}}{1000})\leqslant 1-a.
$$
So
$$
\mathbb{P}(dist(x, X\setminus \tilde{Q}_k)<\ep \delta^{k})\leqslant \mathbb{P}(|x_{k+j}x_{k+j-1}|\geqslant \frac{\delta^{k+j-1}}{1000} \; \forall j=1,\ldots, m) \leqslant (1-a)^{m}\leqslant C\ep^{\eta}
$$
for
$$
\eta=\frac{\log{(1-a)}}{\log(\delta)}.
$$
\end{proof}
\section{Probability to be ``good'' is the same for every cube}
We make the last step to make the probability to be ``good'' not just bounded away from zero, but the same for all cubes.
We use the idea from ~\cite{HM2}.

Take a cube $Q(\omega)$. Take a random variable $\xi_{Q}(\omega^{'})$, which is equally distributed on $[0,1]$. We know that
$$
\mathbb{P}(Q \; \mbox{is good}) = p_{Q}>a>0.
$$
We call $Q$ ``really good'' if
$$
\xi_Q \in [0, \frac{a}{p_Q}].
$$
Otherwise $Q$ joins bad cubes. Then
$$
\mathbb{P}(Q \; \mbox{is really good}) = a,
$$
and we are done.
\section{Application}
As a main application of our construction, we state the following theorem.
\begin{defin}
Let $X$ be a geometrically doubling metric space. Suppose $K(x,y)\colon X\times X \to \R$ is a Calderon-Zygmund kernel of singularity $m$, i. e.
\begin{align}
&|K(x,y)|\leqslant \frac{1}{|xy|^m} \\
&|K(x,y)-K(x', y)|+|K(y,x)-K(y, x')|\leqslant \frac{|xx'|^\ep}{|xy|^{m+\ep}}.
\end{align}
Let $\mu$ be a measure on $X$, such that $\mu(B(x,r))\leqslant C r^{m}$, where $C$ doesn't depend on $x$ and $r$.
We say that $T$ is a Calderon-Zygmund operator with kernel $K$ if
\begin{align}
& T \; \mbox{is bounded} \; L^2(\mu)\to L^2(\mu), \\
& Tf(x) = \int K(x,y)f(y)d\mu(y), \; \forall x\not\in \textup{supp}\mu.
\end{align}
\end{defin}
\begin{defin}
Let $w>0$ $\mu$-a.e. Define
$$
w\in A_{2, \mu}\Leftrightarrow [w]_{2, \mu} = \sup_{B} \frac{1}{\mu(B)} \ili_{B} wd\mu \cdot \frac{1}{\mu(B)} \ili_{B} w^{-1}d\mu < \infty.
$$
\end{defin}
\begin{theorem}[$A_2$ theorem for a geometrically doubling metric space]
Let $X$ be a geometrically doubling metric space, $\mu$ and $T$ as above, $w\in A_{2, \mu}$.
In addition we assume that $\mu$ is a doubling measure. Then
$$
\|T\|_{L^2(wd\mu)\to L^2(wd\mu)} \leqslant C(T)[w]_{2, \mu}.
$$
\end{theorem}
\begin{zamech}
We note that existence of such $\mu$ in a GDMS was proved in ~\cite{KV}.

\end{zamech}

\end{document}